\documentclass[14pt]{amsart}

\usepackage{graphicx}
\usepackage{amssymb}
\usepackage{amsthm}
\usepackage{amsmath}
\usepackage{geometry}
\usepackage{overpic}
\usepackage{color}
\usepackage{multirow}

\usepackage[pdftex,%
            colorlinks,linkcolor=red,citecolor=blue,urlcolor=blue,%
            pdfauthor={Fedor Nilov},%
            pdftitle={Draft},%
            pdfkeywords={circle, web, cyclide},
            pdfstartview=FitH,%
           ]{hyperref}

\newtheorem{theorem}{Theorem}[section]
\newtheorem{lemma}[theorem]{Lemma}
\newtheorem{Main}[theorem]{Main Theorem}
\newtheorem*{Pappus*}{The Pappus Theorem}
\newtheorem*{Brianchon*}{The Brianchon Theorem}
\newtheorem{GrafSauer}[theorem]{The Graf--Sauer Theorem}
\newtheorem{VolkStrub}[theorem]{The Volk--Strubecker Theorem}
\newtheorem{Wunder}[theorem]{The Wunderlich Theorem}
\newtheorem{wund1}[theorem]{The Wunderlich Theorem}

\newtheorem*{Blaschke*}{The Blaschke--Bol Problem}

\theoremstyle{definition}

\newtheorem*{definition*}{Definition}

\newtheorem{problem}[theorem]{Problem}

\theoremstyle{remark}

\begin{document}

\Large

\title{New examples of hexagonal webs of circles}

\author{Fedor Nilov}
\address{Moscow State University}
\email{nilovfk@gmail.com}

\thanks{The author was supported in part by President of the Russian Federation grant MK-3965.2012.1 and by RFBR grant 13-01-12449.}

\begin{abstract} We give several new examples of hexagonal $3$-webs of circles in the plane and
give a survey on such webs. 

\smallskip


\end{abstract}

\maketitle

\Large

\section{Introduction}\label{sec:introduction}

There are many theorems in elementary geometry, in which some natural construction closures with period $6$;
e.g., see \cite[p. 104--107]{Akop}. For example, the famous Pappus and Brianchon theorems can be stated in this way; we give these statements a bit later.
It turns out that such closure theorems come from a beautiful general theory called \emph{web geometry}. It was founded by W. Blaschke and his collaborators
in 1920s; see a nice introduction and references in \cite[\S 18]{Fuchs}. Since then many interesting results have been obtained in the area (see references in \cite{Pottmann}) but many natural questions remained open. In this paper we give several new examples of webs of circles, which is an advance in one of such open questions.

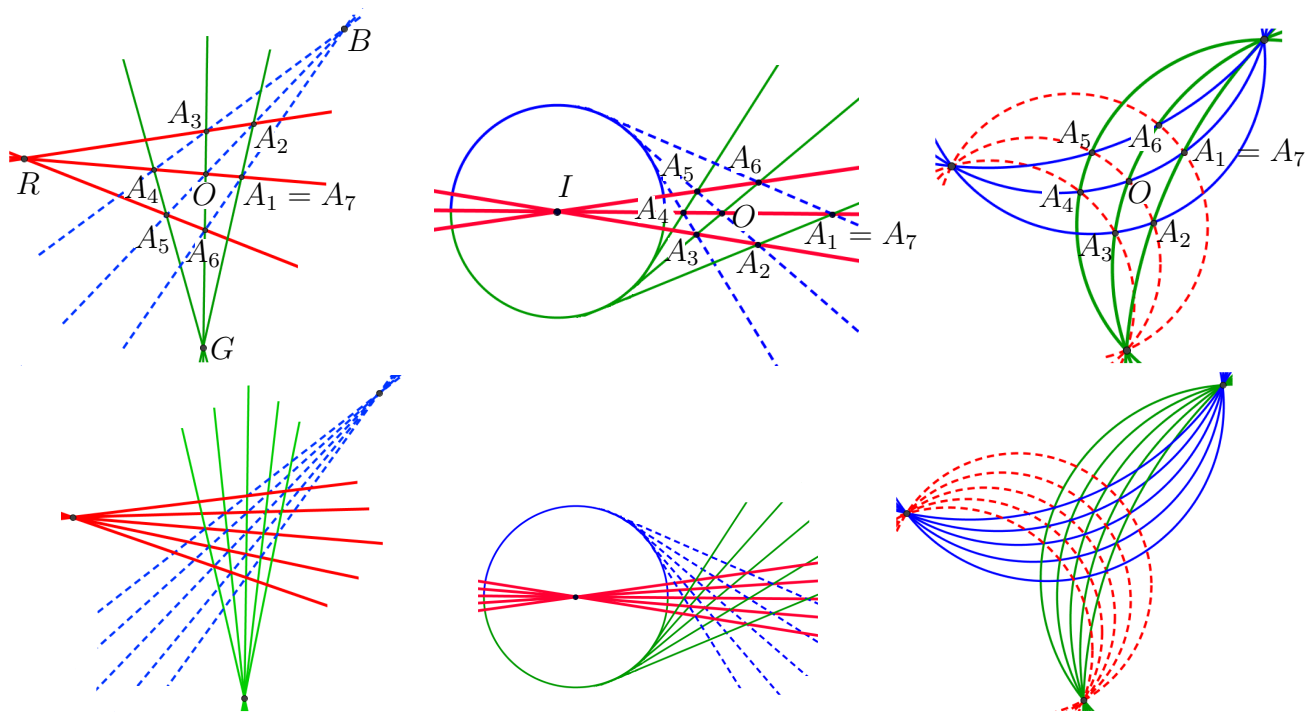
\begin{figure}[hb]
\begin{overpic}[width=0.26\textwidth]{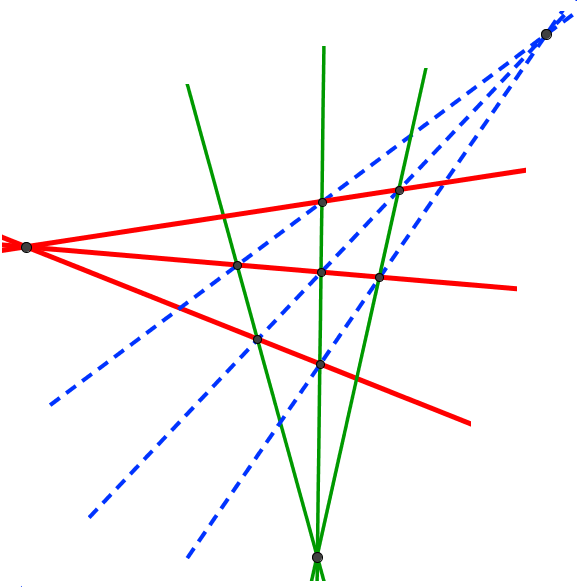}
\put(56, 2){$G$}
\put(2, 48) {$R$}
\put(94, 88) {$B$}
\put(44, 68) {$A_3$}
\put(51, 44) {\color{white}\vrule width .22cm height .35cm }
\put(51, 45) {$O$}
\put(65, 45) {$A_1=A_7$}
\put(69, 60) {\color{white}\vrule width .22cm height .3cm }
\put(68, 61) {$A_2$}
\put(52, 29) {\color{white}\vrule width .25cm height .3cm }
\put(49, 29) {$A_6$}
\put(35, 33) {\color{white}\vrule width .22cm height .3cm }
\put(35, 33) {$A_5$}
\put(32, 47) {\color{white}\vrule width .22cm height .3cm }
\put(32, 47) {$A_4$}
\end{overpic}
\qquad
\begin{overpic}[width=0.31\textwidth]{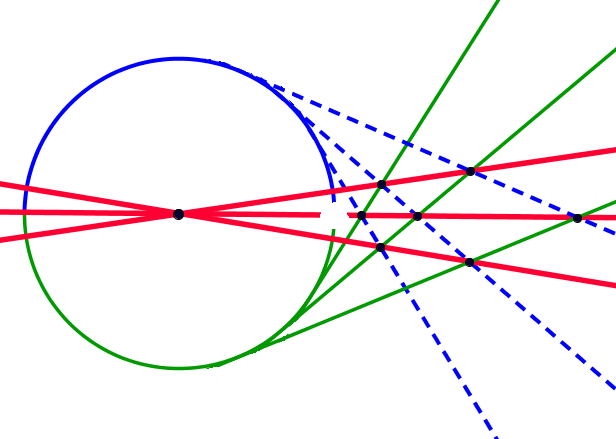}
\put(71, 32) {\color{white}\vrule width .29cm height .35cm }
\put(70, 33){$O$}
\put(29, 40) {$I$}
\put(87, 29) {$A_1=A_7$}
\put(71, 22) {$A_2$}
\put(55, 24) {\color{white}\vrule width .22cm height .35cm }
\put(54, 25) {$A_3$}
\put(48, 34) {\color{white}\vrule width .22cm height .35cm }
\put(47, 35) {$A_4$}
\put(54, 44) {\color{white}\vrule width .22cm height .35cm }
\put(53, 45) {$A_5$}
\put(70, 46) {\color{white}\vrule width .22cm height .35cm }
\put(69, 47) {$A_6$}
\end{overpic}
\qquad
\begin{overpic}[width=0.26\textwidth]{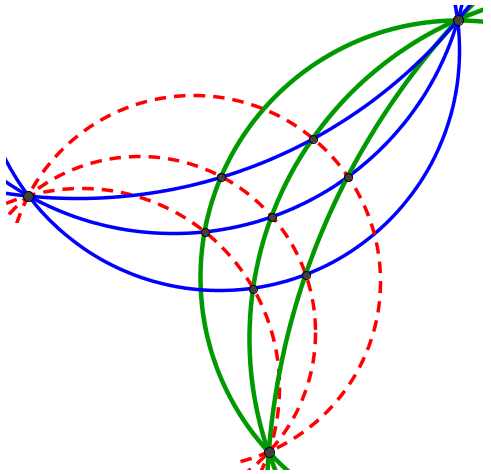}
\put(73, 58) {$A_1=A_7$}
\put(63, 35) {$A_2$}
\put(41, 31) {$A_3$}
\put(31, 44) {\color{white}\vrule width .22cm height .35cm }
\put(30, 45) {$A_4$}
\put(36, 62) {\color{white}\vrule width .22cm height .35cm }
\put(35, 63) {$A_5$}
\put(57, 46) {\color{white}\vrule width .22cm height .3cm }
\put(55, 46) {$O$}
\put(56, 61) {\color{white}\vrule width .3cm height .35cm }
\put(55, 62) {$A_6$}
\end{overpic}
\\
\includegraphics[width=0.25\textwidth]{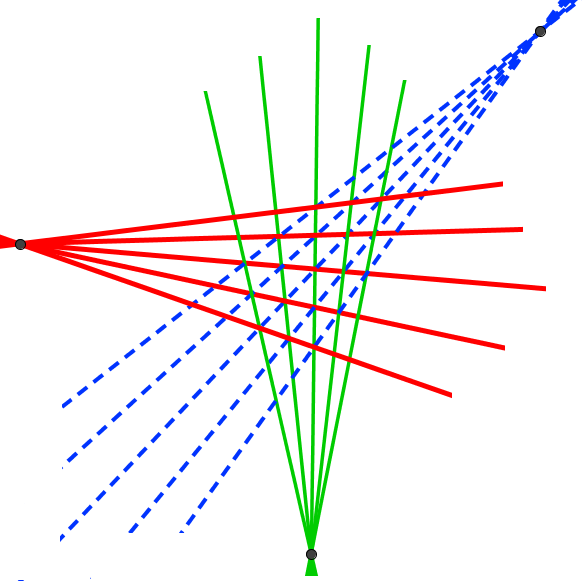}
\qquad
\includegraphics[width=0.25\textwidth]{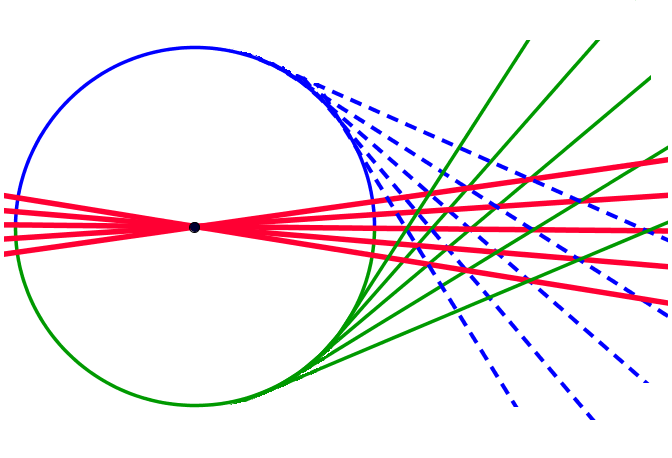}
\qquad
\includegraphics[width=0.25\textwidth]{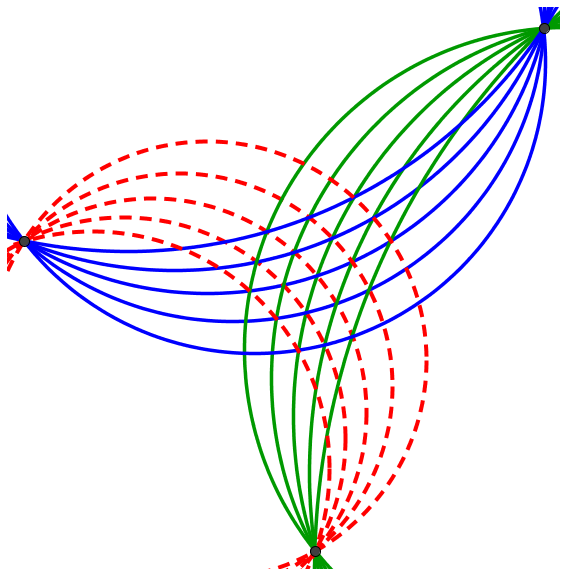}
\caption{Top: Left: \emph{The Pappus theorem.} Middle: \emph{The Brianchon theorem.} Right: \emph{The Blaschke theorem.} Bottom: Left: \emph{The Pappus web} formed by three pencils of lines. Middle: \emph{The Brianchon web} formed by the set of tangent lines to a circle counted twice and a pencil of lines.
Right: \emph{The Blaschke web} formed by three elliptic pencils of circles with the vertices $(R,G)$, $(G,B)$, and $(B,G)$ (see the definition of pencils in \S 1).}
\label{fig1}
\end{figure}

Let us give the statements of the Pappus and Brianchon theorems as closure theorems.

\begin{Pappus*}
A red ($R$), a green ($G$), and a blue ($B$) points are marked in the plane (see Figure~\textup{\ref{fig1}} to the left). Each line passing through exactly one of the marked points is painted the same color as the point.  Take an arbitrary point $O$ inside the triangle $RGB$. Draw the red, the green, and the blue line through the point. On the red line take an arbitrary point $A_1$ inside the triangle $RGB$.
Draw the green line through the point $A_1$. Suppose that the green line intersects the blue line through the point $O$ at a point $A_2$. The green and the blue line through the point $A_2$ have already been drawn; draw the red line through $A_2$. The intersection point of the obtained red line with the green line through the point $O$ is denoted by $A_3$. Continuing this construction we get the points $A_4,A_5,A_6,A_7$. Then  $A_7=A_1$.
\end{Pappus*}


\begin{Brianchon*}
A circle with a point $I$ inside and a point $O$ outside are given; see Figure~\textup{\ref{fig1}} to the middle.
The lines passing through  $I$ are painted red. The rays starting at the points of this circle, tangent to it, and
looking clockwise or counterclockwise are painted green or blue, respectively. 
Construct the points $A_1\ldots A_7$ as in the Pappus Theorem. Then $A_7=A_1$.
\end{Brianchon*}

In this paper we consider a general construction generating theorems of this kind; e.g., see Figure~\textup{\ref{fig1}} to the right.
In what follows by a \emph{circular arc} we mean either a circular arc or a line segment, or a ray, or a line. We assume that all circular arcs do not contain
their endpoints.


Suppose that some circular arcs contained in a domain $\Omega\subset\mathbb{R}^2$ with the endpoints contained in the boundary of
$\Omega$
are painted red, green, and blue. We say that they have \emph{hexagonal property},
if the following $2$ conditions hold (see figures to the left):

\noindent
\begin{tabular}{p{0.2\textwidth}p{0.75\textwidth}}
\vspace{0.0cm}
\centering
\begin{overpic}[width=0.12\textwidth]{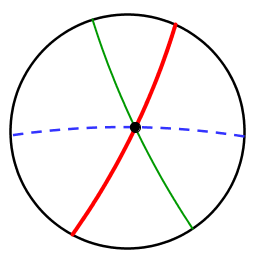}
\put(14,24) {$\Omega$}
\put(58,41) {\color{white}\vrule width .22cm height .35cm }
\put(57,42) {A}
\end{overpic} \\
\begin{overpic}[width=0.21\textwidth]{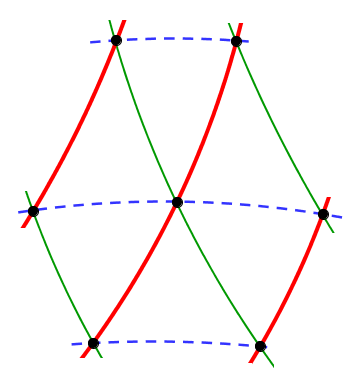}
\put(40,38) {\color{white}\vrule width .36cm height .35cm }
\put(40,38) {$O$}
\put(47,58) {\color{white}\vrule width .36cm height .3cm }
\put(48,59) {\color{red}$\alpha_1$}
\put(56,44) {\color{white}\vrule width .36cm height .3cm }
\put(56,44) {\color{blue}$\gamma_1$}
\put(48,26) {\color{white}\vrule width .36cm height .4cm }
\put(48,28) {\color{green}$\beta_1$}
\put(66,62) {\color{white}\vrule width .36cm height .4cm }
\put(66,64) {\color{green}$\beta_2$}
\put(40,87) {\color{white}\vrule width .36cm height .3cm }
\put(40,88) {\color{blue}$\gamma_2$}
\put(68,21) {\color{white}\vrule width .36cm height .3cm }
\put(68,22) {\color{red}$\alpha_2$}
\put(50,76) {\color{white}\vrule width .6cm height .36cm }
\put(52,77) {$A_1=A_7$}
\put(71,34) {\color{white}\vrule width .36cm height .30cm }
\put(71,36) {$A_2$}
\put(51,9) {\color{white}\vrule width .43cm height .36cm }
\put(51,9) {$A_3$}
\put(35,9) {\color{white}\vrule width .36cm height .3cm }
\put(35,9) {\color{blue}$\gamma_3$}
\put(9,6) {\color{white}\vrule width .5cm height .35cm }
\put(9,9) {$A_4$}
\put(13,25) {\color{white}\vrule width .36cm height .3cm }
\put(13,27) {\color{green}$\beta_3$}
\put(0,34) {\color{white}\vrule width .5cm height .35cm }
\put(0,37) {$A_5$}
\put(15,58) {\color{white}\vrule width .5cm height .35cm }
\put(15,60) {\color{red}$\alpha_3$}
\put(20,76) {\color{white}\vrule width .4cm height .35cm }
\put(20,78) {$A_6$}
\end{overpic}
&

\begin{itemize}
\item \textbf{Foliation condition:}
For each point $A\in \Omega$ there is exactly one circular arc of each color passing through $A$. The arcs of distinct colors either are disjoint or intersect once transversely.
\vspace{0.0cm}
\item
\textbf{Closure condition:}
Consider an arbitrary point $O\in \Omega$. Let $\alpha_1$, $\beta_1$, and $\gamma_1$ be the red, green, and blue circular arcs
passing through $O$, respectively. Consider an arbitrary point $A_1\in\alpha_1$. 
Let $\beta_2$ and $\gamma_2$ be the green and blue circular arcs, respectively, passing through $A_1$. Let
$A_2$ be the intersection point of $\beta_2$ and $\gamma_1$. Consider the red circular arc $\alpha_2$ passing through
$A_2$. Let $A_3$ be the intersection point of $\alpha_2$ and $\beta_1$.
Analogously define the points $A_4$, $A_5$, $A_6$, and
$A_7$. The \emph{hexagonal closure condition} asserts that if all the above points exist, then $A_7=A_1$.
\end{itemize}
\end{tabular}


A trivial example of lines having the hexagonal property is
the lines parallel to the sides of a fixed triangle painted red, green, and blue.
Now if each of the lines intersects a domain $\Omega\subset\mathbb{R}^2$ by at most one line segment
then these segments have the hexagonal closure property. If a real analytic diffeomorphism $f\colon\Omega\rightarrow\Omega'\subset\mathbb{R}^2$ maps
these segments to circular arcs (painted the same color) then the circular arcs
have the hexagonal property as well. This is a motivation for the following definition.

\begin{figure}[hb]
\centering
\begin{overpic}[width=0.4\textwidth]{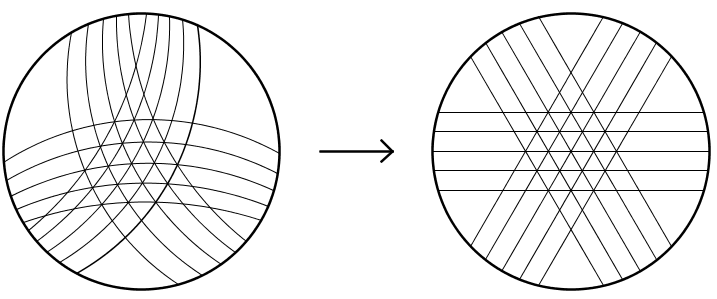}
\put(48, 24) {$f$}
\put(35,35) {$\Omega$}
\put(95,35) {$\Omega'$}
\end{overpic}
\caption{A definition of a hexagonal $3$-web.}
\label{fig2}
\end{figure}





\begin{definition*}
Three sets of circular arcs in a domain $\Omega$
is a \emph{hexagonal $3$-web of circular arcs} (or simply a \emph{web of circular arcs}), if there is
 a real analytic diffeomorphism $f\colon\Omega\rightarrow\Omega'\subset\mathbb{R}^2$ which
takes the sets of arcs to the intersections of the sets of lines
parallel to the sides of a fixed triangle
with the domain $\Omega'$, and all the nonempty intersections of these lines with $\Omega'$ are connected; see Figure~\ref{fig2}.
\end{definition*}

W. Blaschke proved under some regularity assumptions that if some circular arcs have the hexagonal property then they is a hexagonal $3$-web.


We say that three sets of circles in the plane \emph{contain a hexagonal $3$-web}, if their appropriate arcs (possibly empty) is a hexagonal $3$-web in an
appropriate domain.
We allow two of the three sets of circles to coincide; this means that we take two disjoint arcs
from each circle of such set. 
In the latter case we say that the set is \emph{counted twice}; see the example in the Figure~\ref{fig3} and the Brianchon Theorem above.

\begin{figure}[htbp]
\includegraphics[width=0.18\textwidth]{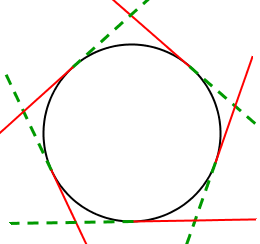}
\caption{The set of tangent lines to a circle is counted twice. We take two disjoint rays from each tangent line to the circle.}
\label{fig3}
\end{figure}

For instance, the following sets of circles 
contain a hexagonal $3$-web; see \cite[p. 19--20]{Blaschke-38} and Figure~\ref{fig1} to the bottom:


(a) \emph{The Pappus web.} Three sets of lines passing through three distinct points $R$, $G$, and $B$, respectively; 

(b) \emph{The Brianchon web.} The set of tangent lines to a conic counted twice and a set of lines passing through a fixed point;

(c) \emph{The Blaschke web.} A set of circles passing through $R$ and $G$, a set of circles passing through $G$ and $B$, and a set of circles passing through $B$ and $G$.



Hexagonal $3$-webs from circular arcs are rare; it is always a luck to find an example. 
In $1938$ W. Blaschke and G. Bol stated the following problem which is still open.

\begin{Blaschke*}
(See \cite[p. 31]{Blaschke-38}.) Find all hexagonal $3$-webs from circular arcs.
\end{Blaschke*}


Let us outline the state of the art; we give precise statements of all known results on such webs in Section~\ref{sec:knownexamples}.
All hexagonal $3$-webs of straight line segments were found by H. Graf and R. Sauer \cite{Blaschke-38}.
All hexagonal $3$-webs of circles belonging to one bundle were found by O. Volk 
and K. Strubecker.
A class of webs of circles generated by a one-parameter group of M\"obius transformations was
considered by W. Wunderlich \cite{Wunderlich}. We give an elementary restatement of his result; see Theorem~\ref{moebius} below.
A highly nontrivial example  of webs of circles doubly tangent to a cyclic was found by W. Wunderlich \cite{Wunderlich}.
For many decades there were no new examples of hexagonal $3$-webs of circular arcs except several webs formed by pencils of circles by V.B. Lazareva, R.S. Balabanova, and H. Erdogan; see \cite{Shelekhov-07}.
Recently, A.M. Shelekhov discussed the classification of all hexagonal $3$-webs formed by pencils of circles \cite{Shelekhov-07}.

Webs of circular arcs on all surfaces distinct from a plane or a sphere are classified in \cite{Pottmann}: N. Lubbes  
proved that any surface in $3$-space containing $\ge 3$ circles through each point
is a so-called Darboux cyclide,
and the webs on the latter are classified in \cite{Pottmann}; see Figure~\ref{fig14} to the left.
There are many examples of webs of conics; e.g., see Figure~\ref{fig14} to the right and references in \cite{Pottmann}.

\begin{figure}[htbp]
\includegraphics[width=0.27\textwidth]{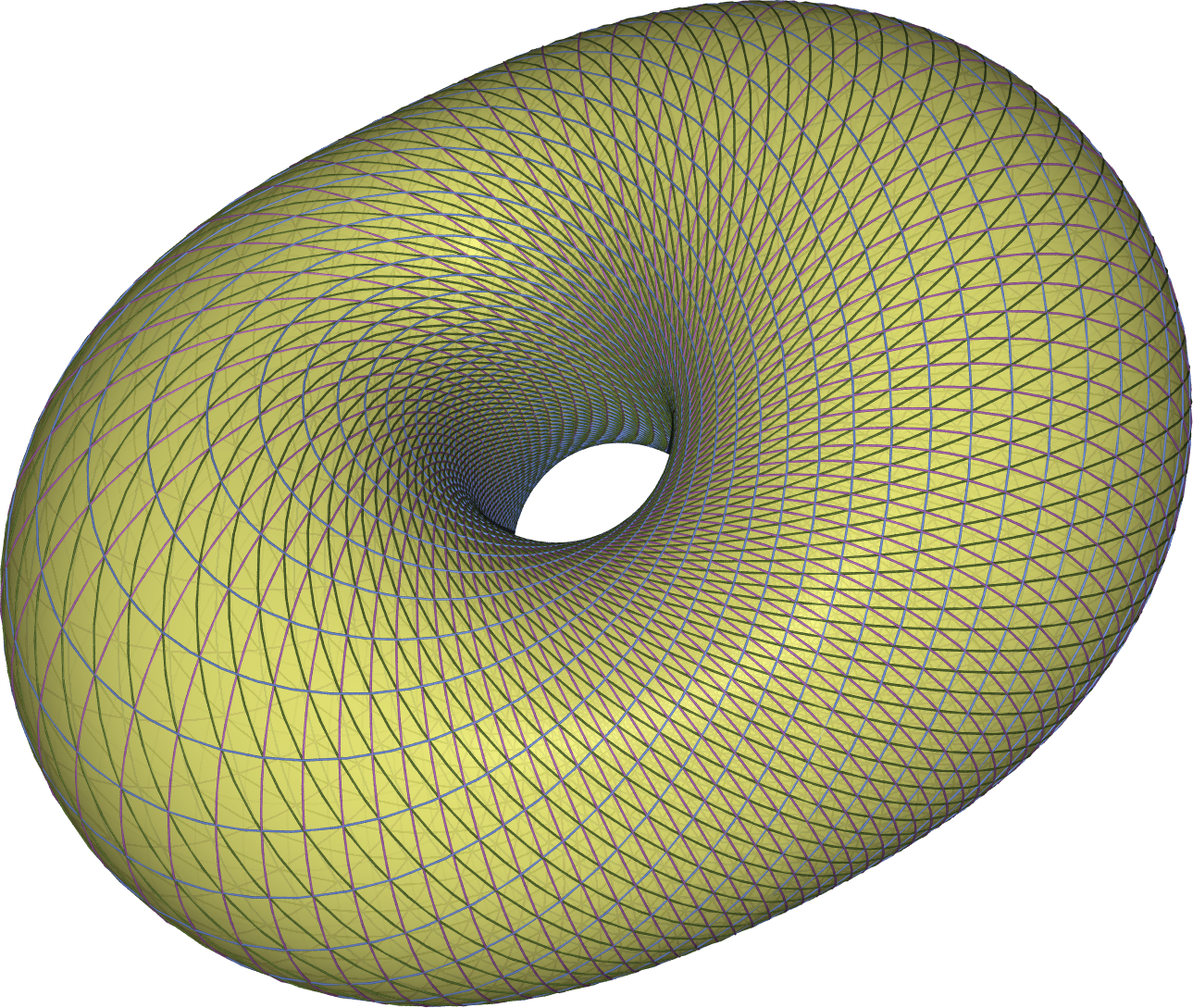}
\qquad
\includegraphics[width=0.25\textwidth]{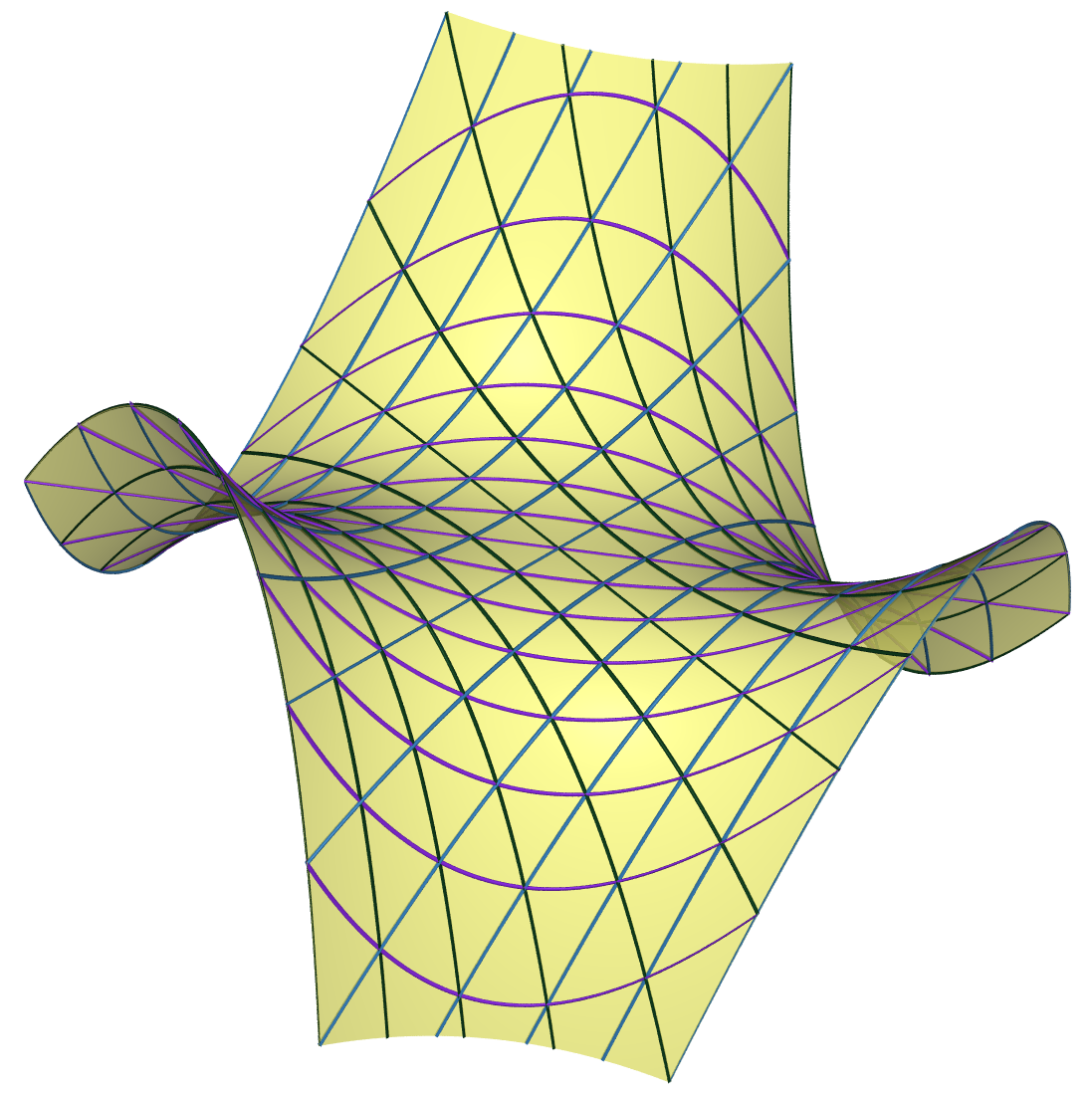}
\caption{Left: A hexagonal $3$-web of circular arcs on the Darboux cyclide \cite{Pottmann}. Right: The surface $z=xy(y-x)$ contains a web of conics,
which are in fact isotropic circles; see Problem~\ref{pr2}.}
\label{fig14}
\end{figure}

Main results of the paper are new examples of webs of circular arcs (see Main Theorem~\ref{newexamples} below). They involve pencils of lines, circles and double tangent circles of conics.


\begin{figure}[htbp]
\includegraphics[width=0.17\textwidth]{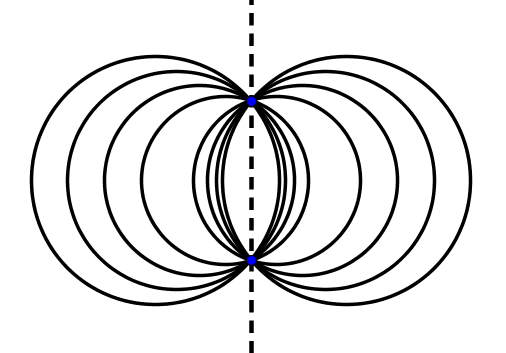}
\qquad
\includegraphics[width=0.17\textwidth]{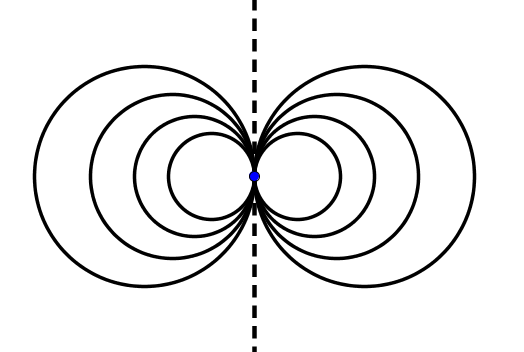}
\qquad
\includegraphics[width=0.17\textwidth]{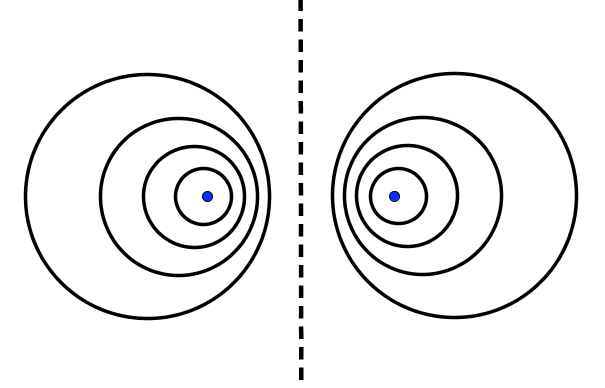}
\caption{Left: An elliptic pencil of circles. Middle: A parabolic pencil of circles. Right: A hyperbolic pencil of circles.}
\label{fig4}
\end{figure}

Let $c_1(x,y)=0$ and $c_2(x,y)=0$ be equations of degree $2$ or $1$ of two distinct circles $c_1$ and $c_2$, respectively.
A \emph{pencil of circles} is the set of all the circles having the equation of the form $\alpha c_1(x,y)+\beta c_2(x,y)=0$, where
$\alpha$ and $\beta$ are real numbers not vanishing simultaneously.
So a \emph{pencil of lines} is a set of all the lines passing through a fixed point (the \emph{vertex} of the pencil) or parallel to a fixed line.
If $c_1$ and $c_2$ are circles with two distinct common points (the \emph{vertices} of the pencil), then all the circles in the pencil pass through these points and the pencil is called \emph{elliptic}; see Figure~\ref{fig4} to the left. If the circles $c_1$ and $c_2$ are tangent (the tangency point then is called the \emph{vertex} of the pencil), then each circle in the pencil is tangent to $c_1$ and $c_2$ at the same point and the pencil is called \emph{parabolic}; see Figure~\ref{fig4} to the middle.
If the circles $c_1$ and $c_2$ have no common points, then the pencil is called \emph{hyperbolic}; see Figure~\ref{fig4} to the right. Any hyperbolic
pencil contains ``circles'' degenerating to points (bold points in the figure). They are called \emph{limiting points} of the pencil.
A pencil of intersecting (respectively, parallel) lines is considered as an elliptic (respectively, parabolic) pencil of circles.
If we take three circles $c_1(x,y)=0$, $c_2(x,y)=0$, and $c_3(x,y)=0$ not belonging to one pencil then
a \emph{bundle of circles} is the set of all the circles having the equation of the form $\alpha c_1(x,y)+\beta c_2(x,y)+\gamma c_3(x,y)=0$.
By a \emph{general conic} we mean either an ellipse distinct from a circle or a hyperbola.

\begin{Main} The following sets of circles contain a hexagonal $3$-web of circular arcs:


(a) The tangent lines to a circle counted twice and a parabolic pencil of circles with the vertex at the center of the circle;

(b) The tangent lines to a general conic counted twice and the hyperbolic pencil of circles with limiting points at the foci of the conic;

(c) The tangent lines to a general conic (counted once), a pencil of lines with the vertex at a focus of the conic, and circles doubly tangent to the conic such that their centers lie on the minor axis of the conic;

(d) The tangent lines to a parabola counted twice and a hyperbolic pencil of circles with limiting points at the focus and
an arbitrary point on the directrix;

(e) The circles doubly tangent to an ellipse with the
eccentricity $\frac{1}{\sqrt{2}}$ counted twice such that their centers lie on the major axis of the ellipse and the elliptic pencil of circles with vertices at the foci of the ellipse.

\label{newexamples}
\end{Main}

\begin{figure}[htbp]
\includegraphics[width=0.7\textwidth]{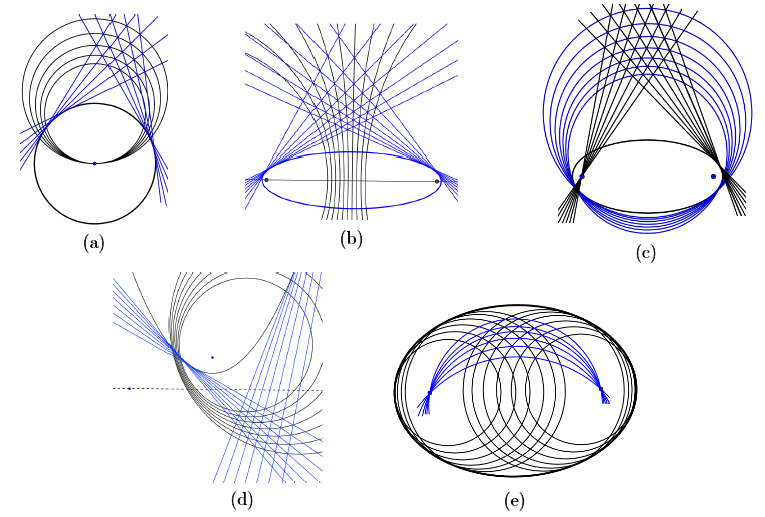} \qquad
\caption{New examples of webs of circular arcs in the plane.
}
\label{fig7}
\end{figure}



The proof of Main Theorem~\ref{newexamples} is elementary and uses only elementary properties of conics.

The paper is organized as follows.
In Section~\ref{sec:knownexamples} we give a survey on webs of circular arcs in the plane. In Section~\ref{sec:newexamples} we prove Main Theorem~\ref{newexamples}.
In Section~\ref{sec:open} we state some open problems.

\section{Known examples}\label{sec:knownexamples}




Let us give precise statements of all known examples of hexagonal $3$-webs of circular arcs in the plane.
Three sets of circles are \emph{transversal}, if for each point from some domain we can find
three circles from distinct sets intersecting transversely at the point. Although we always consider  webs  of \emph{real} circular arcs, in Theorems \ref{moebius}, \ref{wunderlich}, and \ref{shelekhov} we
use some auxiliary \emph{complex} points or circles (we skip their formal definition because it is not used in the proof of main results).



Let $F(a,b,c)$ be a homogeneous polynomial of degree $3$. 
The set of lines $ax+by+c=0$ such that $F(a,b,c)=0$ is called a \emph{set of lines tangent to a curve of class $3$}.

The following theorem characterizes all hexagonal $3$-webs of straight line segments.




\begin{GrafSauer} \textup{\cite[\S 3]{Blaschke-38}} \label{grafsauer} If the lines tangent to a curve of class $3$
counted triply are transversal then they contain a hexagonal $3$-web; see Figure~\ref{fig8} to the left.
Conversely, any hexagonal $3$-web of line segments is contained in such set of lines. 
\label{graf}
\end{GrafSauer}

In the particular case when the polynomial $F(a,b,c)$ above is reducible we get either the Pappus or the Brianchon web.




The following theorem characterizes all hexagonal $3$-webs of circular arcs belonging to one bundle. 
The \emph{Darboux transformation} is the composition of a central projection from a plane in space to a sphere and another central projection from the sphere to the plane such that the center of the second projection belongs to the sphere.
A.G.~Khovanskii proved that Darboux transformations are the only maps of planar domains that take all line segments to circular arcs (see \cite[p. 562]{Timorin}). 

\begin{VolkStrub} \label{VolkStrubecker} The image of any hexagonal $3$-web of straight line segments under a Darboux transformation is a hexagonal $3$-web of circular arcs. Conversely, any hexagonal $3$-web of circular arcs belonging to one bundle
can be obtained by this construction.
\label{volk}
\end{VolkStrub}

\begin{figure}[hb]
\includegraphics[width=0.20\textwidth]{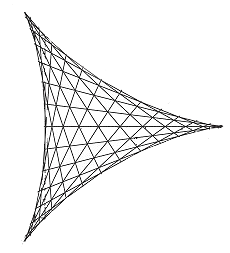} \qquad
\includegraphics[width=0.22\textwidth]{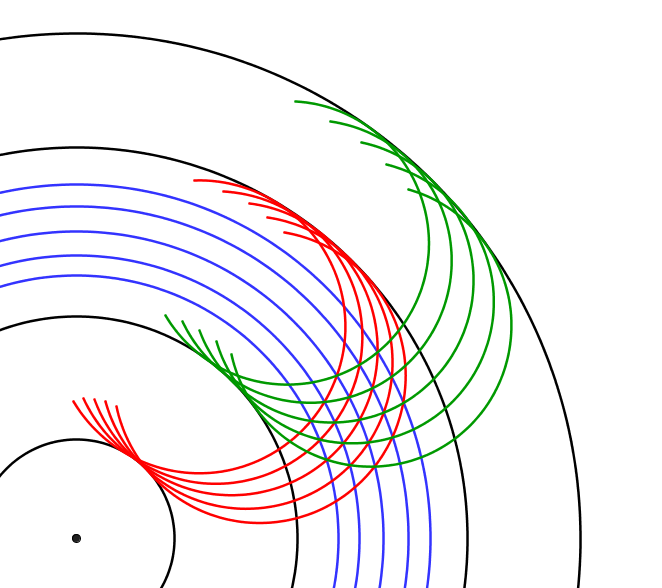} \qquad
\includegraphics[width=0.20\textwidth]{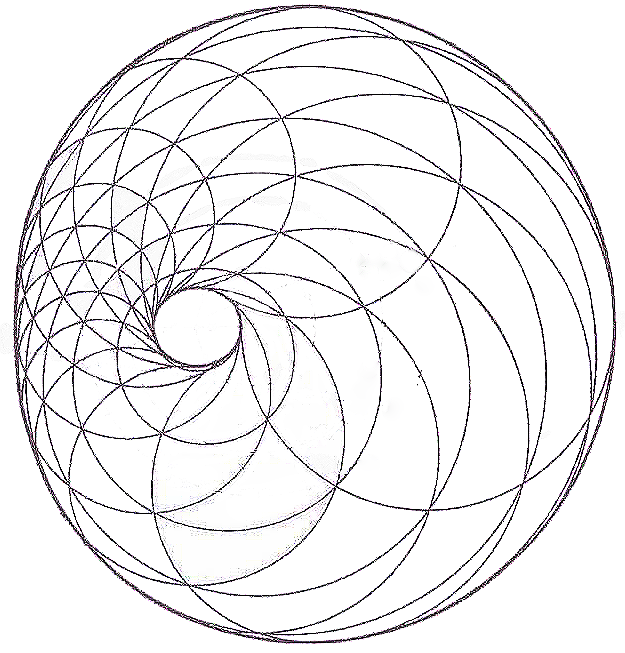}
\caption{Left: A hexagonal $3$-web of lines tangent to a deltoid which is a particular case of a curve of class $3$. Middle: Generation of a hexagonal $3$-web of circles using a one-parametric group of rotations. Right: A hexagonal $3$-web of circles doubly tangent to a cyclic.}
\label{fig8}
\end{figure}


W. Wunderlich considered the following example of a hexagonal $3$-web of circular arcs.

\begin{wund1}
Let two circles $\omega_1$ and $\omega_2$ in the plane have a common point $O$ of transversal intersection.
Let a one-parametric group $\mathcal{M}_t$ of M\"obius transformations of the plane be such that
 each orbit is either a circle or a point. Suppose that the orbit of $O$ intersects transversely $\omega_1$ and $\omega_2$. Then the circles $\{\mathcal{M}_t(\omega_1)\}$, $\{\mathcal{M}_t(\omega_2)\}$, and the orbits of $\mathcal{M}_t$ contain a hexagonal $3$-web; see Figure~\ref{fig8} to the middle.
\label{wund1}
\end{wund1}

In Section~\ref{sec:newexamples} we show
that this construction is equivalent to the following elementary one.





An \emph{Apollonian set} of a pencil of circles is either a set of circles tangent to two distinct (possibly complex or null) circles from the pencil, or a parabolic pencil with the vertex at a vertex of the pencil, or a hyperbolic pencil
with limiting points at the vertices of the pencil.

\begin{theorem} \label{moebius} If a pencil of circles and two Apollonian sets of this pencil
are transversal then they
contain a hexagonal $3$-web.
\end{theorem}

Let us give the most wonderful example of webs of circular arcs.

A \emph{cyclic} is the curve given by an equation of the form
\begin{equation}\label{Cyclic}
a(x^2+y^2)^2+(x^2+y^2)(bx+cy)+Q(x,y)=0,
\end{equation}
where $a,b,c$ are constants and $Q(x,y)$ is a polynomial of degree at most $2$ not vanishing simultaneously.
Note that conics, lima\c{c}ons of Pascal, and Cartesian ovals are particular cases of cyclics.

We say that a circle is \emph{tangent} to a cyclic, if the circle either has a real tangency point with the cyclic, or a complex one, or passes through a singular point of the cyclic.
The circles doubly tangent to a cyclic naturally split into $\le 4$ \emph{families}: the centers of circles from one family
lie on one conic or line; see \cite[Remark 16]{Pottmann}. 

\begin{Wunder}\textup{\cite{Wunderlich}} \label{wunderlich} 
If three distinct families
of circles doubly tangent to
a cyclic are transversal then they contain a hexagonal $3$-web; see Figure~\ref{fig8} to the right.
\end{Wunder}




A particular case (stated by Blaschke in 1953) of the Blashke-Bol problem was to find all triples of pencils of circles containing webs.

\begin{theorem} \textup{\cite{Shelekhov-07}} \label{shelekhov} The following  pencils of circles contain a hexagonal $3$-web:

(a) \textup{(Volk-Strubecker)} Three pencils of circles belonging to the same bundle. 

(b) \textup{(Lazareva)} Three hyperbolic pencils with a common complex circle such that in each of the pencils there is a circle
orthogonal to all the circles of the other two pencils. 

(c) \textup{(Lazareva)} Two elliptic pencils and one hyperbolic pencil with a common real circle such that in each of the pencils
there is a circle orthogonal to all the circles of the other two pencils.

(d) \textup{(Balabanova)} Two orthogonal pencils and the third pencil having a common circle with each of the two orthogonal pencils. 

(e) \textup{(Balabanova)} Two orthogonal parabolic pencils and one hyperbolic pencil; one of the limiting points of the hyperbolic
pencil coincides with the common vertex of the parabolic pencils.

(f) \textup{(Blaschke)} Three elliptic pencils of circles with the vertices $(A,B)$, $(B,C)$, and $(C,A)$. 

(g) \textup{(Erdogan)} Two elliptic pencils with the vertices $(A,B)$ and $(B,C)$ and the hyperbolic pencil with the limiting points $C$ and $A$.


(h) \textup{(Lazareva)} Two parabolic pencils and an elliptic pencil with the vertices
at the vertices of the parabolic pencils.



(j) \textup{(Erdogan)} An elliptic pencil with the vertices $A$ and $B$, the hyperbolic pencil
with the limiting points $B$ and $C$, and a parabolic pencil with the vertex at $A$ such that the common circle of the elliptic and hyperbolic
pencils is orthogonal to the circle passing through the points $A$, $B$, and $C$.

\end{theorem}

According to A.M.~Shelekhov \cite{Shelekhov-07} these are all the possible triples of pencils of circles containing a hexagonal $3$-web. (He divides some of examples (a)--(j) into several subclasses.)

We see that our examples of hexagonal $3$-webs of circular arcs in Main Theorem~\ref{newexamples}(a)-(e)  are indeed new; see Table~\ref{table1}.

\begin{table}[hb]
\begin{center} \caption{\label{table1}}
\begin{tabular}{|p{3.1cm}|p{14cm}|}
\hline
Known examples & Difference from the new examples \\
\hline
Theorem~\ref{graf},~\ref{volk}  & the circles in each example of Theorem~\ref{newexamples}(a)-(e) do not belong to one bundle \\
\hline
Theorem~\ref{moebius} & Theorem~\ref{newexamples}(a): the lines tangent to the circle do not belong to the given parabolic pencil; \\ & Theorem~\ref{newexamples}(b)-(e): there is a family of circles enveloping a conic distinct
from a circle 
 \\
\hline
Theorem~\ref{wunderlich} & the envelope of all the circles in examples Theorem~\ref{newexamples}(a)-(e) is not one cyclic\\
\hline
Theorem~\ref{shelekhov} & the circles in each example Theorem~\ref{newexamples}(a)-(e) do not belong to $3$ pencils\\
\hline
\end{tabular}
\end{center}

\end{table}





\section{Proofs}\label{sec:newexamples}


In the proofs below we construct a real analytic diffeomorphism $f\colon\Omega\rightarrow\Omega'\subset\mathbb{R}^2$ which maps the intersections of circles from the given sets with an appropriate domain $\Omega$ to the segments of the lines $x=\mathrm{const}$, $y=\mathrm{const}$ and $x+y=\mathrm{const}$.

We need the notions of left and right tangent lines from a point to a conic.
Let a point and a conic be given. Consider a line passing through the point not intersecting the conic. Let us start to rotate this line around the given point
counterclockwise. Suppose that there are two moments when this line is either tangent to the conic or is an asymptotic line.
We say that the lines at the first and the second moments (if such moments exist) are called the \emph{left} and the \emph{right} tangent lines, respectively. If the point lies on the conic then by definition the left and the right tangent lines coincide with the ordinary tangent line.

Let $d(X, \lambda)$ be the distance from a point $X$ to a line $\lambda$.
By $\angle{(\alpha, \beta)}\in[0,\pi)$ we denote the oriented angle between lines $\alpha$ and $\beta$.

\emph{Proof of Theorem~\ref{newexamples}(a).} This assertion is a limiting case of Theorem~\ref{newexamples}(b)
in which the foci of the conic converge to each other. The given conic converges to a circle. The hyperbolic pencil of circles with limiting points at the foci converges to a parabolic pencil of circles with the vertex at the center of the circle.
Thus point (a) follows from (b) because the foliation condition is clearly satisfied and passing to the limit respects the hexagonal closure condition.
Theorem~\ref{newexamples}(a) is proved modulo Theorem~\ref{newexamples}(b).

\smallskip

\emph{Proof of Theorem~\ref{newexamples}(b).} Denote by $\gamma$ the given general conic.
Denote by $F_1$ and $F_2$ the foci of $\gamma$.
Let $U$ be an appropriate domain such that for each point $A\in U$ the left and the right tangent lines $\alpha(A)$ and $\beta(A)$
to the conic $\gamma$ passing through $A$,
and the circle $\omega(A)$ from the pencil passing through $A$ exist and intersect transversely.


\begin{figure}[hb]
\centering
\begin{overpic}[width=0.25\textwidth]{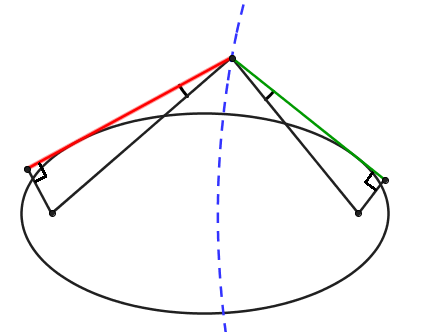}
\put(50, 32) {\color{white}\vrule width .3cm height .4cm}
\put(48,35) {$\omega(A)$}
\put(15,24) {$F_1$}
\put(0,40) {$P$}
\put(74,24) {$F_2$}
\put(92,40) {$Q$}
\put(58,67) {$A$}
\put(34, 53) {\color{white}\vrule width .3cm height .3cm}
\put(20,56) {$\alpha(A)$}
\put(67, 52) {\color{white}\vrule width .3cm height .3cm}
\put(67,54) {$\beta(A)$}
\put(29,3) {\color{white}\vrule width .3cm height .3cm}
\put(30,3) {$\gamma$}
\end{overpic}
\caption{To the proof of Lemma~\ref{cl1}.}
\label{fig9}
\end{figure}
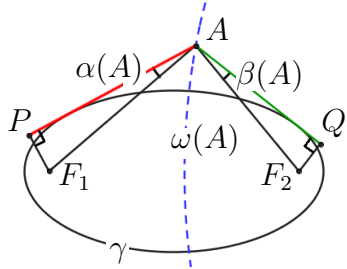

\begin{lemma} 
For a fixed circle $\omega$ from the pencil the ratio $\frac{d(F_1, \alpha(A))}{d(F_2, \beta(A))}$ does not depend on the point $A\in\omega\cap U$.
\label{cl1}
\end{lemma}

\begin{proof}
Let $P$ be the orthogonal projection of $F_1$ onto the line $\alpha(A)$. Let $Q$ be the orthogonal projection
of $F_2$ onto the line $\beta(A)$; see Figure~\ref{fig9}.
By the \emph{isogonal property of conics} (see \cite[\S 1.4]{AkopZasl}), we have $\angle{PAF_1}=\angle{F_2AQ}$.
Thus the right triangles $F_1AP$ and $F_2AQ$ are similar. By the well-known \emph{geometric characterization of a
hyperbolic pencil of circles} (see \cite[Theorem 2.12]{AkopZasl}) the ratio $\frac{|F_1A|}{|F_2A|}$ does not depend on the point $A$
lying on a circle $\omega$ from the pencil with limiting points $F_1$ and $F_2$.
Hence

\smallskip
$$\frac{d(F_1, \alpha(A))}{d(F_2, \beta(A))}=\frac{|F_1P|}{|F_2Q|}=\frac{|F_1A|}{|F_2A|}=\mathrm{const}.$$
\smallskip



\end{proof}

Consider the map $f$: $U\rightarrow\mathbb{R}^2$ such that for each point $A\in U$

\smallskip
$$f(A):=(\ln{d(F_1,\alpha(A))}, -\ln{d(F_2, \beta(A))}).$$
\smallskip

Choose a subdomain (still denoted by $U$) in which the differential of $f$ is nonzero. 
From Lemma~\ref{cl1} it follows  that $f$ maps the intersections of the left tangent lines to $\gamma$, the right tangent lines to $\gamma$, and the circles from the pencil  with $U$ to the segments of the lines  $x=\mathrm{const}$, $y=\mathrm{const}$, and $x+y=\mathrm{const}$. In particular, three transversal curves
$\alpha(A)$, $\beta(A)$, and $\omega(A)$ have transversal $f$-images. So the differential of $f$ in $U$ is nondegenerate because it is nonzero. Thus the restriction of the map $f$
to an appropriate subdomain $\Omega\subset U$ is a diffeomorphism. 
Theorem~\ref{newexamples}(b) is proved.

\smallskip

\emph{Proof of Theorem~\ref{newexamples}(c).} Denote by $\gamma$ the given conic. Denote by $F_1$ the given focus of $\gamma$. 
For each point $A$ denote by $\alpha(A)$ the line belonging to the pencil and passing through $A$. Denote by $\beta(A)$ the left tangent line to $\gamma$ passing through $A$. 



Let $(\phi_1, \phi_2)\subset[0,\pi/2)$ be an interval
such that for each $\phi\in(\phi_1, \phi_2)$ there is a point
$T$ on the conic $\gamma$ satisfying $\angle{(\alpha(T), \beta(T))}=\phi$.
Let $U$ be an appropriate domain such that for each point $A\in U$ there exist $\alpha(A)$, $\beta(A)$ and
$\angle{(\alpha(A), \beta(A))}\in(\phi_1, \phi_2)$.




\begin{lemma} The locus of all the points $A\in U$ such that $\angle{(\alpha(A), \beta(A))}=\phi,$ where $\phi\in(\phi_1, \phi_2)$ is fixed, is an arc of the circle doubly tangent to the conic $\gamma$ such that the center of this circle lies on the minor axis
of the conic.
\label{locus}
\end{lemma}

\begin{proof} 


Denote by $P$ the projection of the focus $F_1$ onto the line containing $\beta(A)$; see Figure~\ref{fig10} to the left.
By \emph{pedal circle property of conics} (see \cite[\S 11.8]{Akop}) the projections of the focus $F_1$ of the conic $\gamma$ onto the tangent lines of $\gamma$ lie on one circle $\omega_0$. The center of the circle $\omega_0$ coincides with the center of the conic.


\begin{figure}[hb]
\begin{overpic}[width=0.25\textwidth]{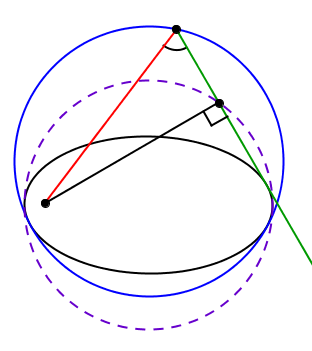}
\put(31, 89) {\color{white}\vrule width .35cm height .3cm}
\put(32,90) {$\omega$}
\put(14, 64) {\color{white}\vrule width .3cm height .4cm}
\put(15,65) {$\omega_0$}
\put(15,35) {$F_1$}
\put(52,92) {$A$}
\put(64,72) {$P$}
\put(29, 64) {\color{white}\vrule width .3cm height .3cm}
\put(29,65) {$\alpha(A)$}
\put(81, 28) {\color{white}\vrule width .3cm height .4cm}
\put(81,30) {$\beta(A)$}
\put(29,20) {\color{white}\vrule width .3cm height .3cm}
\put(30,21) {$\gamma$}
\end{overpic}
\qquad
\begin{overpic}[width=0.29\textwidth]{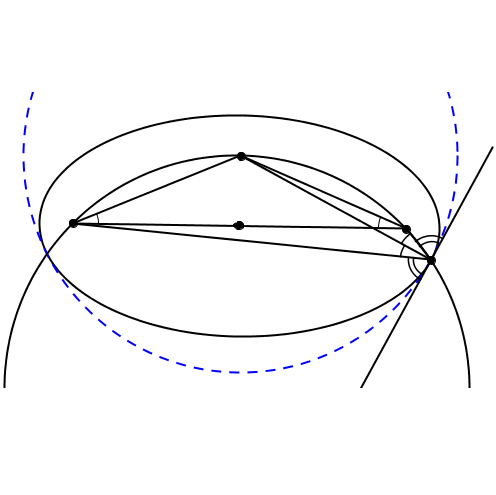}
\put(9, 57) {\color{white}\vrule width .35cm height .4cm}
\put(8,58) {$F_1$}
\put(82, 55) {\color{white}\vrule width .35cm height .35cm}
\put(81,56) {$F_2$}
\put(52,67) {\color{white}\vrule width .1cm height .1cm}
\put(46,70) {$O_{\omega}$}
\put(45,57) {$O_{\gamma}$}
\put(89,41) {\color{white}\vrule width .35cm height .3cm}
\put(89,42) {$T$}
\put(72,23) {\color{white}\vrule width .35cm height .45cm}
\put(73.5,25.5) {$\beta(T)$}
\put(30,32) {\color{white}\vrule width .3cm height .3cm}
\put(31,32) {$\gamma$}
\put(39,23) {\color{white}\vrule width .35cm height .3cm}
\put(40,24) {$\omega$}
\end{overpic}
\caption{To the proof of Lemma~\ref{locus}.}
\label{fig10}
\end{figure}
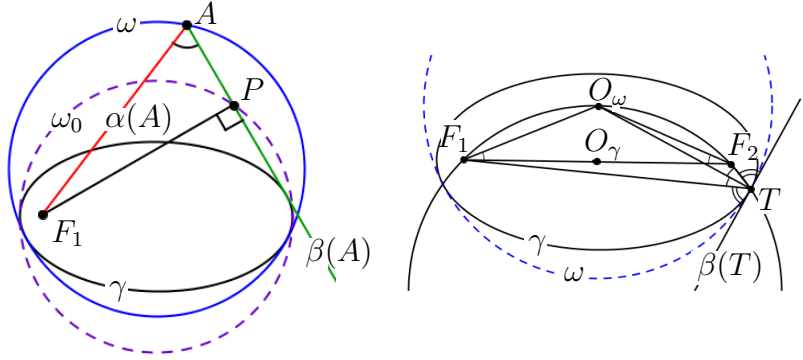

Note that $|F_1A|/|F_1P|=1/|\sin{\angle{(\alpha(A), \beta(A))}}|=1/|\sin{\phi}|$ and $\angle(F_1P,F_1A)=\frac{\pi}{2}-\angle{(\alpha(A),\beta(A))}=\frac{\pi}{2}-\phi$.
The image of the circle $\omega_0$ under the composition of the counterclockwise rotation about $F_1$ through the angle $\frac{\pi}{2}-\phi$
and the homothety with the center $F_1$ and the coefficient of dilation $1/|\sin{\phi}|$ is a circle $\omega$.
Evidently, the given locus is $\omega \cap U$.

Let us prove that $\omega$ is doubly tangent to $\gamma$ and the center of $\omega$ lies on the minor
axis of $\gamma$. Denote by 
$O_{\gamma}$ the center of $\gamma$ and by $O_{\omega}$ the center of $\omega$; see Figure~\ref{fig10} to the right. 
Since $\phi\in(\phi_1, \phi_2)$, it follows that there is $T\in\gamma$ such that $\angle{(\alpha(T),\beta(T))}=\phi$. 
Let $F_2$ be the other focus of $\gamma$.
Since $\omega$ is the image of $\omega_0$ under the above composition,
we have $\angle{F_1O_{\gamma}O_{\omega}}=\angle{F_1PA}=\frac{\pi}{2}$ and $\phi=\angle{F_1AP}=\angle{F_1O_{\omega}O_{\gamma}}$. So the point $O_{\omega}$
lies on the bisector of the segment $F_1F_2$, $\angle{(F_1F_2, F_1O_{\omega})}=\angle{(F_2O_{\omega}, F_2F_1)}$, and
$\angle{F_1O_{\omega}F_2}=2\phi$.
By the construction of the point $T$ and the \emph{optical property of conics} (See \cite[\S 1.3]{AkopZasl}) we have $\phi=\angle{(TF_1, \beta(T))}=\angle{(\beta(T), TF_2)}$. Thus $\angle{(O_{\omega}F_1,O_{\omega}F_2)}=\angle{(TF_1, TF_2)}$. So the points $F_1$, $O_{\omega}$, $F_2$, and $T$ are cocyclic. Since $\angle{(F_1F_2, F_1O_{\omega})}=\angle{(F_2O_{\omega}, F_2F_1)}$ and the points $F_1$, $O_{\omega}$, $F_2$, and $T$ are cocyclic we have $\angle{(TO_{\omega}, TF_1)}=\angle{(TF_2, TO_{\omega})}$. So $O_{\omega}T$ is perpendicular to $\beta(T)$. Hence, $\beta(T)$ is the tangent line to $\omega$ at $T$. Thus $T$
is a tangency point of $\omega$ and $\gamma$.
By reflection symmetry, the point $T'$ symmetric to $T$ with respect to the minor axis of $\gamma$ is another tangency point of $\omega$ and $\gamma$.
Thus $\omega$ and $\gamma$ are doubly tangent.
\end{proof}


Denote by $\omega(A)$ the locus of all the points $X\in U$ such that $\angle{(\alpha(X), \beta(X))}=\angle{(\alpha(A), \beta(A))}$.
Let $\lambda$ be the major axis of the conic $\gamma$.


Consider the map $f$: $U\rightarrow\mathbb{R}^2$ such that for each point $A\in U$

\smallskip
$$f(A):=(\angle{(\alpha(A),\lambda)}, \angle{(\lambda,\beta(A))}).$$
\smallskip

Choose a subdomain (still denoted by $U$) in which $f$ is continuous, the differential of $f$ is nonzero, and for each point $A\in U$
the curves $\alpha(A)$, $\beta(A)$, and $\omega(A)$ are transversal. By Lemma~\ref{locus} it follows  that $f$ maps the intersections of lines from the pencil, the left
tangent lines to $\gamma$, and considered circles doubly tangent to $\gamma$ with $U$ to the segments of the lines  $x=\mathrm{const}$, $y=\mathrm{const}$ and $x+y=\mathrm{const}$. In particular, three transversal curves
$\alpha(A)$, $\beta(A)$, and $\omega(A)$ have transversal $f$-images. So the differential of $f$ in $U$ is nondegenerate because it is nonzero. Thus the restriction of the map $f$
to an appropriate domain $\Omega\subset U$ is a diffeomorphism. Theorem~\ref{newexamples}(c) is proved.








\smallskip

\emph{Proof of Theorem~\ref{newexamples}(d).} Denote by $\gamma$ the given parabola. Denote by $F$ and $\delta$ the focus and the directrix of the parabola.
Consider the hyperbolic pencil with the limiting points $F$ and $L\in\delta$.
Let $U$ be an appropriate domain such that for each point $A\in U$ the left and the right tangent lines $\alpha(A)$ and $\beta(A)$
to the parabola $\gamma$ passing through $A$,
and the circle $\omega(A)$ from the pencil passing through $A$ exist and intersect transversely.



\begin{figure}[htbp]
\centering
\begin{overpic}[width=0.35\textwidth]{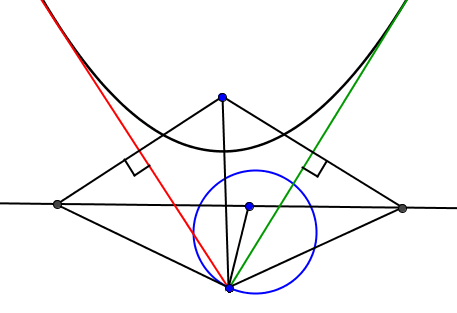}
\put(58,4) {\color{white}\vrule width .6cm height .3cm}
\put(59,3) {$\omega(A)$}
\put(50,49) {$F$}
\put(44,0) {$A$}
\put(9,16) {$P$}
\put(86,16) {$Q$}
\put(2.5, 19) {\color{white}\vrule width .25cm height .35cm}
\put(3,22) {$\delta$}
\put(55,25) {$L$}
\put(23, 39) {\color{white}\vrule width .3cm height .3cm}
\put(14,39) {$\alpha(A)$}
\put(72, 40) {\color{white}\vrule width .3cm height .35cm}
\put(72,40) {$\beta(A)$}
\put(40,33) {\color{white}\vrule width .3cm height .3cm}
\put(40,34) {$\gamma$}
\end{overpic}
\caption{To the proof of Lemma~\ref{cl3}.}
\label{fig11}
\end{figure}
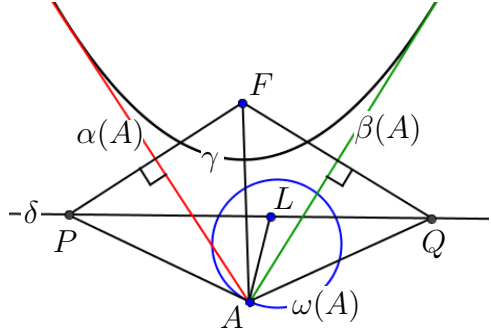

Consider the line passing through $F$ and perpendicular
to $\alpha(A)$; see Figure~\ref{fig11}.
Denote by $P$ the intersection of this line and $\delta$. Consider the line passing through $F$ and perpendicular
to $\beta(A)$. Denote by $Q$ the intersection of this line and $\delta$.

Set

\smallskip
$$s(A):=\frac{|PL|\cdot|\cos{\angle{(\alpha(A),\delta)}}|}{|FP|} \textup{\quad and \quad} t(A):=\frac{|QL|\cdot|\cos{\angle{(\delta, \beta(A))}}|}{|FQ|}.$$
\smallskip


\begin{lemma}
For a fixed circle $\omega$ from the pencil the product $s(A) \cdot t(A)$ does not depend on the point $A\in\omega\cap U$.
\label{cl3}
\end{lemma}

\begin{proof} 
By a well-known property of a parabola, \emph{the point $A$ is the center of the circle circumscribed about $\triangle FPQ$} (see \cite[Lemma 1.2]{AkopZasl}).
So $|AF|=|AP|=|AQ|=R$, where $R$ is a radius of the circle circumscribed about $\triangle FPQ$.
By the sine theorem we have

$$2 R=\frac{|FP|}{|\sin{\angle{(QF, QP)}}|}=\frac{|FP|}{|\cos{\angle{(\delta,\beta(A))}}|} \textup{ and }
2 R=\frac{|FQ|}{|\sin{\angle{(PQ, PF)}}|}=\frac{|FQ|}{|\cos{\angle{(\alpha(A),\delta)}}|}.$$
By the well-known property of a power of a point with respect to a circle we have
$|AL|^2=R^2-|PL|\cdot|QL|.$
Then

\smallskip
$$s(A) \cdot t(A)=\left(\frac{|PL|\cdot|\cos{\angle{(\alpha(A),\delta)}}|}{|FP|}\right)\cdot\left(\frac{|QL|\cdot|\cos{\angle{(\delta, \beta(A))}}|}{|FQ|}\right)=$$
$$=\frac{|PL|\cdot |QL|}{4 R^2}=\frac{1}{4}\left(1-\left(\frac{|AL|}{|AF|}\right)^2\right)=\mathrm{const},$$
\smallskip
where 
the fourth equality follows from the geometric characterization of a hyperbolic
pencil of circles (see the proof of Theorem~\ref{newexamples}(b)).

\end{proof}

Consider the map $f$: $U\rightarrow\mathbb{R}^2$ such that for each point $A\in U$

\smallskip
$$f(A):=(\ln{s(A)}, \ln{t(A)}).$$
\smallskip

Choose a subdomain (still denoted by $U$) in which the differential of $f$ is nonzero. 
From Lemma~\ref{cl3} it follows  that $f$ maps the intersections of the left tangent lines to $\gamma$, the right tangent lines to $\gamma$, and the circles from the pencil  with $U$ to the segments of the lines  $x=\mathrm{const}$, $y=\mathrm{const}$, and $x+y=\mathrm{const}$. In particular, three transversal curves
$\alpha(A)$, $\beta(A)$, and $\omega(A)$ have transversal $f$-images. So the differential of $f$ in $U$ is nondegenerate. Thus the restriction of the map $f$
to an appropriate subdomain $\Omega\subset U$ is a diffeomorphism. 
Theorem~\ref{newexamples}(d) is proved.

\smallskip

\emph{Proof of Theorem~\ref{newexamples}(e).} Denote by $\gamma$ the given ellipse. Denote by $F_1$, $F_2$, and $O_{\gamma}$ the foci and the center of $\gamma$.
Consider the elliptic pencil of circles with the vertices $F_1$ and $F_2$.
Consider the Cartesian coordinate system such that $O_{\gamma}$ is the origin and
the line containing the major axis $\lambda$ is the $Ox$-axis.
Without loss of generality assume that the minor axis of $\gamma$ has length $2$.
Let a circle with the center on the major axis $\lambda$ be doubly tangent to $\gamma$.
The line passing through the tangency points separates the circle into two circular arcs:
the ``left'' tangent circular arc and the ``right'' tangent circular arc. 
Let $U$ be an appropriate domain such that for each point $A\in U$ the left and the right tangent circular arcs $\alpha(A)$ and $\beta(A)$
to $\gamma$ passing through $A$,
and the circle $\omega(A)$ from the pencil passing through $A$ intersect transversely,
the center $O_{\alpha}$ of the circle containing $\alpha(A)$ ``lies to the left''
of $O_{\gamma}$, the center $O_{\beta}$ of the circle containing $\beta(A)$ ``lies to the right'' of $O_{\gamma}$, and also $s(A) > t(A)$, where
$s(A):=|O_{\alpha}O_{\gamma}|$ and $t(A):=|O_{\beta}O_{\gamma}|$; see Figure~\ref{fig12}.


\begin{lemma} For a fixed circle $\omega$ the product $\frac{1-s^{2}(A)}{s^{2}(A)} \cdot \frac{1-t^{2}(A)}{t^{2}(A)}$ does not depend on the point $A\in\omega\cap U$.
\label{st}
\end{lemma}

\begin{proof} 
Let $P$ and $Q$ be points on the circles containing $\alpha(A)$ and $\beta(A)$, respectively, such that $O_{\alpha}P$ and $O_{\beta}Q$ are perpendicular to $\lambda$. 
Consider the circle $\omega_0$ such that the minor axis of the ellipse $\gamma$ is a diameter of $\omega_0$.
Since the eccentricity of $\gamma$ is equal to $\frac{1}{\sqrt{2}}$ we get that the foci $F_1$ and $F_2$ lie on $\omega_0$.

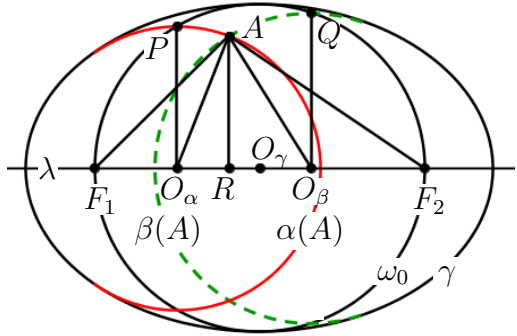
\begin{figure}[hb]
\centering
\begin{overpic}[width=0.37\textwidth]{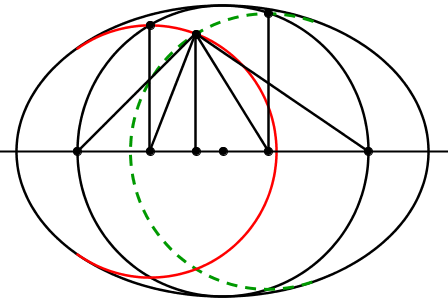}
\put(58, 4) {\color{white}\vrule width .3cm height .3cm}
\put(15,25) {\color{white}\vrule width .3cm height .4cm}
\put(15,26) {$F_1$}
\put(80,25) {\color{white}\vrule width .3cm height .4cm}
\put(80,26) {$F_2$}
\put(46, 60) {\color{white}\vrule width .3cm height .25cm}
\put(46,60) {$A$}
\put(27,56) {\color{white}\vrule width .3cm height .33cm}
\put(27,56) {$P$}
\put(61,58) {\color{white}\vrule width .3cm height .37cm}
\put(61,59) {$Q$}
\put(6, 30) {\color{white}\vrule width .25cm height .35cm}
\put(6,32) {$\lambda$}
\put(53, 18) {\color{white}\vrule width .6cm height .5cm}
\put(53,20) {$\alpha(A)$}
\put(25, 18) {\color{white}\vrule width .6cm height .5cm}
\put(25,20) {$\beta(A)$}
\put(30,28) {$O_{\alpha}$}
\put(56,27) {\color{white}\vrule width .4cm height .35cm}
\put(56,28) {$O_{\beta}$}
\put(48,36) {$O_{\gamma}$}
\put(40,28) {$R$}
\put(84,12) {\color{white}\vrule width .4cm height .3cm}
\put(85,12) {$\gamma$}
\put(72,11) {\color{white}\vrule width .4cm height .35cm}
\put(73,12) {$\omega_0$}
\end{overpic}
\caption{To the proof of Lemma~\ref{st}.}
\label{fig12}
\end{figure}

\begin{lemma} The points $P$ and $Q$ lie on the circle $\omega_0$.
\label{cl4}
\end{lemma}

\begin{proof}
We are going to prove that for each point $Y\in\omega_0$ the circle $\omega$ with the center at the projection $X$
of $Y$ onto $\lambda$ and the radius $|XY|$ is doubly tangent to $\gamma$.
Let $(a,0)$ be coordinates of $X$.
Then $|XY|=\sqrt{1-a^2}$.
So the circle $\omega$ 
has the equation of the form $(x-a)^2+y^2=1-a^2$. We can rewrite this equation in the form
$g(x,y,a)=0$, where $g(x,y,a)=(x-a)^2+y^2+a^2-1$. Let us compute the envelope of this family of circles. It is well-known that the equation of such envelope
can be obtained from the following system

$$\textup{ (1) } g(x,y,a)=0 \textup{ and } \textup{ (2) } \frac{\partial }{\partial a}g(x,y,a)=0$$
\smallskip
by eliminating $a$. We have $\frac{\partial }{\partial a}g(x,y,a)=-2x+4a$. By $(2)$ we have $a=x/2$. Substituting $a=x/2$ in $(1)$, we have
the envelope equation $\frac{x^2}{2}+y^2=1$. This is the equation of $\gamma$. Thus $\omega$ is tangent to $\gamma$.
Since the center of  $\omega$ lies on the major axis, $\omega$ is doubly tangent to $\gamma$.
\end{proof}

Let us continue the proof of Lemma~\ref{st}.
Further we denote $s(A)$ and $t(A)$ simply by $s$ and $t$.

From Lemma~\ref{cl4} it follows that $|AO_{\alpha}|=|PO_{\alpha}|=\sqrt{1-s^{2}}$ and $|AO_{\beta}|=|QO_{\beta}|=\sqrt{1-t^{2}}$.
Let $R$ be the projection of $A$ onto $\lambda$. Then $|RO_{\alpha}|^2-|RO_{\beta}|^2=|AO_{\alpha}|^2-|AO_{\beta}|^2=t^2-s^2$.
Since $|O_{\alpha}O_{\beta}|=s+t$ we have $|RO_{\alpha}|=t$ and $|RO_{\beta}|=s$. 
Hence $|AR|=\sqrt{O_{\alpha}A^2-O_{\alpha}R^2}=\sqrt{O_{\alpha}P^2-O_{\alpha}R^2}=\sqrt{1-s^2-t^2}$,
$|F_1R|=1-s+t$, and $|F_2R|=1+s-t$. Thus $|F_1A|^2=1-s^2-t^2+(1-s+t)^2=2(1+(t-s)-st)$ and $|F_2A|^2=1-s^2-t^2+(1+s-t)^2=2(1+(s-t)-st)$.

For each point $A\in\omega\cap U$ we have $\angle{F_1AF_2}=\mathrm{const}$.
By the cosine theorem we have

$$\cos{\angle{F_1AF_2}}=\frac{|F_1F_2|^2-|F_1A|^2-|F_2A|^2}{2|F_1A|\cdot |F_2A|}=\frac{4-2(1+(t-s)-st)-2(1+(s-t)-st)}{\sqrt{4(1+(t-s)-st)\cdot (1+(s-t)-st)}}=$$
$$=\frac{2st}{\sqrt{(1-s^2)(1-t^2)}}.$$
\smallskip

Thus $\frac{1-s^{2}(A)}{s^{2}(A)} \cdot \frac{1-t^{2}(A)}{t^{2}(A)}=\frac{4}{\cos^2{\angle{F_1AF_2}}}=\mathrm{const}$. Lemma~\ref{st} is proved.

Consider the map $f$: $U\rightarrow\mathbb{R}^2$ such that for each point $A\in U$

\smallskip
$$f(A):=\left(\ln{\frac{1-s^{2}(A)}{s^{2}(A)}}, \ln{\frac{1-t^{2}(A)}{t^{2}(A)}}\right).$$
\smallskip

Choose a subdomain (still denoted by $U$) in which the differential of $f$ is nonzero. 
From Lemma~\ref{st} it follows  that $f$ maps the intersections of the left tangent circular arcs to $\gamma$, the right tangent circular arcs to $\gamma$, and the circles from the pencil with $U$ to the segments of the lines  $x=\mathrm{const}$, $y=\mathrm{const}$, and $x+y=\mathrm{const}$. In particular, three transversal curves
$\alpha(A)$, $\beta(A)$, and $\omega(A)$ have transversal $f$-images. So the differential of $f$ in $U$ is nondegenerate. Thus the restriction of the map $f$
to an appropriate subdomain $\Omega\subset U$ is a diffeomorphism. 
Theorem~\ref{newexamples}(e) is proved.
\end{proof}


\begin{proof}
[Proof that Theorem~\ref{moebius} is equivalent to the Wunderlich Theorem 2.3.]
It is well-known that \emph{any one-parameter group of M\"obius transformations is conjugate to a one-parameter group of either dilations, or rotations, or translations, or loxodromic transformation.}

The orbits of a one-parameter group of loxodromic transformations are not circular arcs. Thus
it suffices to consider the following three cases.

Case 1. \emph{$\mathcal{M}_t$ is conjugate to a one-parameter group of dilations.} The orbits of the group
are arcs of circles belonging to one elliptic pencil. It's easy to verify that there are the following
three possibilities: 1) there are two distinct (possibly complex) circles from the pencil
tangent to  $\omega_i$; 2)  $\omega_i$ passes through a vertex of the pencil; 3)  $\omega_i$ belongs to a hyperbolic pencil with limiting points at
the vertices of the pencil.
If a circle $\omega_i$
passes through a vertex of the pencil then the images of $\omega_i$  are circles belonging to the parabolic pencil
with the same vertex. If $\omega_i$ belongs to a hyperbolic pencil with limiting points at
the vertices of the pencil then the images of
$\omega_i$ are circles belonging to the hyperbolic pencil. Overwise the images of
$\omega_i$ are circles tangent to two distinct (possibly complex) fixed circles from the pencil.

Case 2. \emph{$\mathcal{M}_t$ is conjugate to a one-parameter group of rotations.}
The orbits of the group are arcs of circles belonging to one hyperbolic pencil. 
The images of $\omega_i$ are circles tangent to two distinct (possibly null) fixed circles from the pencil.

Case 3. \emph{$\mathcal{M}_t$ is conjugate to a one-parameter group of translations.}
The orbits of the group are arcs of circles belonging to one parabolic pencil.
If a circle $\omega_i$ passes through the vertex of the pencil then the images of $\omega_i$  are circles belonging to the parabolic pencil with the same vertex. If $\omega_i$ does not pass through a vertex of the pencil then the images of
$\omega_i$ are circles tangent to two distinct fixed circles from the pencil.

\end{proof}

\smallskip

\section{Open problems}\label{sec:open}

The following open problems may be a good warm-up before attacking the Blaschke--Bol Problem.

\begin{problem} \emph{Web Transformation Problem.} Prove that in Theorem~\ref{newexamples}(b) the hyperbolic pencil can be replaced by the elliptic one with the vertices at the limiting points of the initial pencil.
Analogous replacement is not possible for Theorem~\ref{newexamples}(d).
\label{pr1}
\end{problem}

\begin{figure}[htbp]
\includegraphics[width=0.23\textwidth]{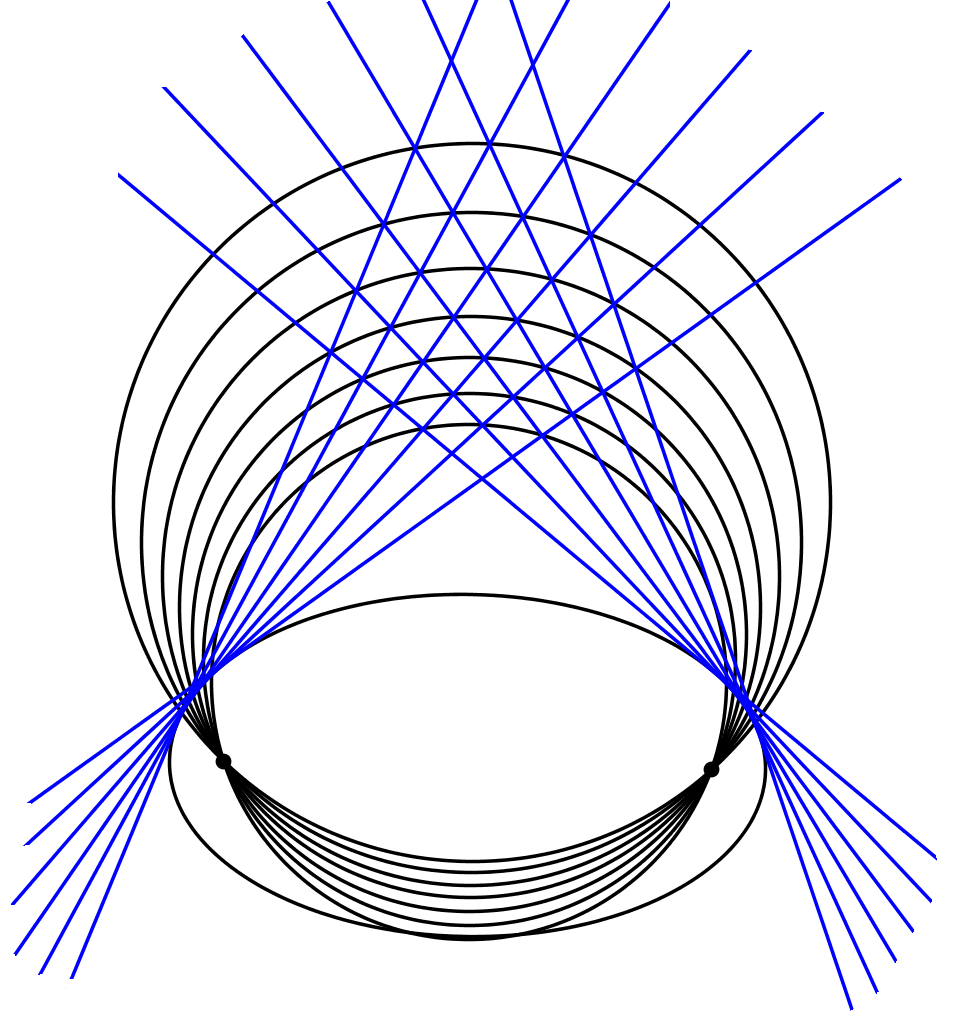}
\caption{To Problem~\ref{pr1}.}
\label{fig13}
\end{figure}


A \emph{cubic series} is a family of circles $a(t)(x^2+y^2)+b(t)x+c(t)y+d(t)=0$, where $a(t)$, $b(t)$, $c(t)$, and $d(t)$ are polynomials of degree $\le 3$.

\begin{problem} \emph{Cubic Series.} Prove that the set of circles $(1-t^3)(x^2+y^2)+2(1+t)x+2(t^2+t^3)y-1-t^3=0$, where $t\in\mathbb{R}$, counted triply, contains a web. Find all cubic series that contain a web.
\label{cubic}
\end{problem}

It is easy to see that examples of webs in Problem~\ref{pr1} and Problem~\ref{cubic} are not particular cases of the examples considered in $\S 1,2$.

\begin{problem} \emph{Complement Problem.} \textup{(A.A. Zaslavsky and I.I. Bogdanov, private communication)}
Given two sets of lines
$x=\mathrm{const}$ and $y=\mathrm{const}$,
find all sets of circles which together with them contain a web.
\end{problem}

Our last problem concerns $3$-dimensional isotropic geometry; see references in \cite{Pottmann}.
An \emph{isotropic circle} is either a parabola
with the axis parallel to $Oz$ or an ellipse whose projection onto the plane $Oxy$ is a circle.
An \emph{isotropic plane} is a plane parallel to $Oz$.
An \emph{isotropic sphere} is a paraboloid of revolution with the axis parallel to $Oz$.


\begin{problem} \emph{Webs of isotropic circles.} Find all webs of isotropic circles in the isotropic plane and on all surfaces except planes and isotropic spheres; see Figure~\ref{fig14} to the right.
\label{pr2}
\end{problem}



\section{Acknowledgements}\label{sec:acknowledgements}

The author is very grateful to M.~Skopenkov for the motivation and constant attention to this work,
and to O.~Karpenkov, A.~Shelekhov, S.~Tabachnikov, V.~Timorin for valuable remarks. 







\end{document}